\documentclass[english]{ourlematema}
\usepackage[all,cmtip]{xy}
\usepackage[english]{babel}
\usepackage{amssymb}
\usepackage{graphicx}
\usepackage[T1]{fontenc}
\usepackage[utf8]{inputenc}
\usepackage{slashbox}

\usepackage{hhline}
\usepackage{xcolor}
\usepackage{shuffle}
\usepackage{wasysym}
\usepackage{amsmath,amsthm,amssymb,verbatim}
\usepackage{tikz}

\usepackage{float} 
\usepackage{caption}
\usepackage[pdftex, pdfusetitle, plainpages=false,  bookmarks, bookmarksnumbered,colorlinks, linkcolor=blue, citecolor=blue,filecolor=black, urlcolor=blue]{hyperref}

\usepackage{enumerate}
\usepackage{mathtools}
\usepackage{kbordermatrix}

\renewcommand{\:}{\colon}
\newcommand{\p}{\mathbb{P}}

\newcommand{\C}{\mathbb{C}}
\newcommand{\R}{\mathbb{R}}
\newcommand{\T}{\mathcal{T}}

\newcommand{\rk}{\mathop{\rm rank}\nolimits}

\newcommand{\Sym}{\mathop{{\rm Sym}}\nolimits}

\newcommand{\SO}{\mathop{\rm SO}\nolimits}
\newcommand{\GL}{\mathop{\rm GL}\nolimits}

\newtheorem{pro}[thm]{Proposition}
\newtheorem{lem}[thm]{Lemma}   

\theoremstyle{definition}

\newtheorem{es}[thm]{Example}

\newtheorem{rmk}[thm]{Remark}

\newcommand{\proj}[0]{\operatorname{Proj}}

\newcommand{\reg}{\mathrm{Reg}}
\newcommand{\irr}{\mathrm{Irr}}

\DeclareMathOperator{\Gr}{Gr}

\renewcommand{\kbldelim}{(}
\renewcommand{\kbrdelim}{)}

\title{On the eigenpoints of cubic surfaces}


\newcommand{\at}{\makeatletter @\makeatother}

\author{T. Ö. Çelik}
\address{%
Max Planck Institute for Mathematics in the Sciences, Leipzig, Germany\\
\email{celik\at mis.mpg.de}
}
\author{F. Galuppi}
\address{%
Max Planck Institute for Mathematics in the Sciences, Leipzig, Germany\\
\email{galuppi\at mis.mpg.de}
}
\author{A. Kulkarni}
\address{%
Technische Universit\"at Kaiserslautern,
Kaiserslautern, Germany\\
\email{avinash\at mis.mpg.de}
}

\author{M.-\c S. Sorea}
\address{%
Max Planck Institute for Mathematics in the Sciences, Leipzig, Germany\\
\email{sorea\at mis.mpg.de}
}

 \date{2020/04/13}

\begin{document}

\maketitle

\begin{abstract}
We show that the eigenschemes of $4 \times 4 \times 4$ symmetric tensors are parameterized by a linear subvariety of the Grassmannian $\Gr(3,\p^{14})$. We also study the decomposition of the eigenscheme into the subscheme associated to the zero eigenvalue and its residue. In particular, we describe the possible degrees and dimensions.
\end{abstract}

\section{Introduction}
The goal of this paper is to study the eigenpoints of order three tensors. The spectral theory of tensors is a multi-linear generalization of the study of eigenvalues, singular values, eigenvectors and singular vectors in the case of matrices. Starting with the works of Qi \cite{qi05} and Lim \cite{lim}, there has been steady progress and strong interest in the subject, both theoretically and in the applications to hypergraph theory, data analysis, automatic control, magnetic resonance imaging, higher order Markov chains, and optimization \cite{qi1, landsberg}.

Given a tensor $\mathcal{T}\in(\C^{n+1})^{\otimes 3}$ and a matrix $A\in\C^{n+1}\otimes\C^{n+1}$, the vector~$\T\cdot A$, defined by
\[(\T\cdot A)_k=\sum_{i=0}^n\sum_{j=0}^n \T_{ijk}A_{ij}
\]
is called the tensor contraction of $\T$ and $A$ with respect to the first and second axes. An analogous definition can be given for a different choice of two of the three axes. Therefore, a choice of two axes induces a quadratic map
	\[
	\begin{array}{rcccc}
		\mathcal{T} \: & \C^{n+1} & \rightarrow & \C^{n+1} \\
		& x & \mapsto & \mathcal{T} \cdot (x \otimes x).
	\end{array}
	\]An \emph{eigenvector} of the tensor $\mathcal{T}$, with respect to the chosen directions, is a non-zero vector $x \in \C^{n+1}$ such that $\mathcal{T}\cdot(x\otimes x) = \lambda x$ for some $\lambda \in \C$, and an \emph{eigenpoint} is the associated equivalence class in $\p^n$. As we point out in Definition~\ref{def:eigenscheme}, the condition of being an eigenpoint can be expressed as the vanishing of minors of a suitable matrix, thus giving the eigenpoints a scheme structure. This closed subscheme of $\p^n$ is the \emph{eigenscheme} of $\T$ with respect to the chosen axes. 

In this article, we fix the contraction to be along the first two directions and we use the terminology of eigenvector, eigenpoint, eigenscheme with the implicit reference to these axes. However, we stress that there are interesting relations between the eigenschemes associated to different directions of the same tensor; this phenomenon of eigencompatibility was studied in detail by Abo, Seigal, and Sturmfels in \cite[Section 3]{AboSeigalSturmfels}. 

With a fixed choice of axes, we may assume without loss of generality that $\mathcal{T} \in \Sym^2 \C^{n+1} \otimes \C^{n+1}$. Indeed, let $\tilde{\T}\in(\C^{n+1})^{\otimes 3}$ be the tensor obtained from $\T$ by switching the first and the second index. In other words, $(\tilde{\T})_{ijk}=\T_{jik}$. Then
\begin{align*}
(\tilde{\T}\cdot (x\otimes x))_k=\sum_{i=0}^n\sum_{j=0}^n \tilde{\T}_{ijk}x_ix_j
=\sum_{j=0}^n\sum_{i=0}^n \T_{jik}x_jx_i=(\T\cdot (x\otimes x))_k,
\end{align*}
hence $\T$ and $\tilde{\T}$ induce the same quadratic map. We refer to elements of $\Sym^2\C^{n+1} \otimes \C^{n+1}$ as \emph{partially symmetric tensors}. A \emph{symmetric tensor} is an element of $\Sym^3 \C^{n+1}$. The space $\Sym^3 \C^{n+1}$ of symmetric tensors is canonically isomorphic to the space $\C[x_0,\dots,x_n]_3$ of homogeneous cubic polynomials. The symmetric tensor $\T$ corresponds to the polynomial
\[
f=\sum_{i_0+\ldots+i_n=3} \T_ {i_0\ldots i_n}x_{i_0}\cdot\ldots\cdot x_{i_n}\in\C[x_0,\dots,x_n]_3.
\]
Equivalently, $\mathcal{T}$ defines a polynomial $f$ by $f(x)=\T \cdot (x \otimes x \otimes x)$.
Conversely, given a homogeneous cubic form $f$ in ${n+1}$ variables, each of the partial derivatives of $f$ is a quadratic form, which in turn can be viewed as an element of $\Sym^2 \C^{n+1}$; the tuple of quadratic polynomials
	$
		\left( \frac{1}{n} \frac{\partial f}{\partial x_0},	\dots, \frac{1}{n} \frac{\partial f}{\partial x_n} \right)
	$,
viewed as $(n+1) \times (n+1)$ symmetric matrices, defines a tensor in $\Sym^3 \C^{n+1}$. 

In Section \ref{sec:prelim} we provide the necessary definitions and the setup. 
In Section~\ref{sec: deg and dim} we look at the decomposition of the eigenscheme into the subscheme of eigenpoints with eigenvalue $0$ (the \emph{irregular eigenpoints}) and its residue (the \emph{regular eigenpoints}); we study the dimensions of the components in this decomposition. In Section~\ref{sec: zero dimensional} we focus on the degree of 0-dimensional regular eigenschemes of ternary and quaternary cubics. In Section~\ref{sec: grassmanian} we show that there is a natural bijection between $3$-planes in $\p^{14}$ satisfying linear constraints and eigenschemes of cubic surfaces.

\section{Preliminaries}\label{sec:prelim}

Fix a projective space $\p^n$ and denote by $\C[x_0, \ldots, x_n]$ its coordinate ring. We also fix $\p^{n+1} := \operatorname{Proj} \C[x_0,\ldots, x_n, \lambda]$. For convenience, we denote $x := (x_0,\ldots, x_n)$.

A partially symmetric tensor $\mathcal{T} \in \Sym^2 \C^{n+1} \otimes\C^{n+1}$ can be viewed as a tuple of $n+1$ quadratic forms $(q_0(x),\ldots, q_n(x))$ given by the contraction $\mathcal{T} \cdot~(x \otimes x)$, similar to how a symmetric tensor $\mathcal{T}$ defines a cubic form in $n+1$ variables by the contraction ${\mathcal{T} \cdot (x \otimes x \otimes x)}$. Equivalently, the quadratic forms are those associated to the $n+1$ symmetric matrices of size $(n+1)\times (n+1)$ obtained by slicing $\T$.

\begin{dfn}\label{def:eigenscheme}
Let $\mathcal{T} = (q_0, \ldots, q_n)$ be a partially symmetric tensor. Define the \emph{scheme of eigenpairs of $\mathcal{T}$} by
	\[
	\tilde{E}(\T)=V( q_0(x) - \lambda x_0, \ldots, q_n(x) - \lambda x_n)\subset\p^{n+1}=\operatorname{Proj} \C[x_0,\ldots, x_n, \lambda].
	\]
Observe that $[0,\dots,0,1]\in\tilde{E}(\T)$.  Let $\pi:\p^{n+1}\dashrightarrow \p^n$ be the projection from $[0,\dots,0,1]$. The image of $\pi$ is a closed subscheme of $\p^n$. The \emph{eigenscheme} of $\mathcal{T}$, denoted by $E(\T)$, is the image under $\pi$ of the residue of $ \tilde{E}(\T)$ with respect to $[0,\dots,0,1]$. 
Equivalently, $E(\T)\subset\p^n$ is
		the common vanishing set of the $2 \times 2$ minors of
			\[\left(
			\begin{matrix}
			x_0 & x_1 & \dots & x_n \\
			q_0 & q_1 & \dots & q_n
			\end{matrix}\right).
			\]
When $\T$ is symmetric, we can consider it as a homogeneous polynomial $f$. In this case $q_i=\frac{\partial f}{\partial x_i}$, and we  denote its eigenscheme by $E(f)$.
\end{dfn}

The first question we can ask about eigenpoints is whether they always exist. If so, we would like to know how many of them are there.

\begin{lemma}\label{lem:nonempty} Let $\T\in\Sym^2\C^{n+1}\otimes\C^{n+1}$ be a partially symmetric tensor.
\begin{enumerate}
\item \label{bullet: count general eigenpoints} If $\T$ is general, then ${E}(\T)$ consists of $2^{n+1}-1$ reduced points.
\item \label{bullet: at least an eigenpoint} $[0,\dots,0,1]$ is a smooth isolated point for $\tilde{E}(\T)$, and moreover $E(\T)\neq~\emptyset$. In particular, every smooth cubic polynomial has at least a regular eigenpoint.
\end{enumerate}
\begin{proof}
\begin{enumerate}
\item By B\'ezout's theorem, it suffices to show that the quadrics defining the eigenscheme are transverse. Since transversality is an open condition, it is enough to exhibit one example of a partially symmetric tensor with $2^{n+1}-1$ reduced eigenpoints. It is easy to check that the symmetric tensor $x_0^3+\ldots+x_n^3$ satisfies this requirement.
\item To be smooth and to be isolated are local properties, so we can work in the affine space $\C^n$ defined by $\lambda=1$. In this chart the point $[0,\dots,0,1]$ is the origin $p=(0,\dots,0)$ and the Jacobian of $\tilde{E}(\T)$ is the $(n+1)\times(n+1)$ matrix
\[J=
\left( \begin{matrix}
\frac{\partial q_0}{x_0}-1 & \frac{\partial q_0}{x_1} & \dots &\frac{\partial q_0}{x_n}\\
\frac{\partial q_1}{x_0} & \frac{\partial q_1}{x_1}-1 &  \ddots &\frac{\partial q_1}{x_n}\\
\vdots  & \ddots & \ddots & \frac{\partial q_{n-1}}{x_n}\\
\frac{\partial q_n}{x_0} & \frac{\partial q_n}{x_1} & \frac{\partial q_n}{x_{n-1}}&\frac{\partial q_n}{x_n}-1
\end{matrix}\right).
\]
Since $q_0,\dots,q_n$ are homogeneous, so are their derivatives. This implies that, up to a sign, $J(p)=J(0,\dots,0)$ is the identity matrix, hence it has maximal rank. This proves that $\tilde{E}(\T)$ is smooth at $p$. The tangent space to $\tilde{E}(\T)$ at $p$ is defined by $J(p)\cdot (x_0,\dots,x_n)^\top=0$, so it has equations $x_0=\ldots=x_n=0$. Therefore $T_p \tilde{E}(\T)=\{p\}$, hence $p$ is an isolated point for $\tilde{E}(\T)$.

Since $E(\T)$ is the image of $\tilde{E}(\T)$ under the projection from $p$, in order to show that it is not empty it is enough to show that $\tilde{E}(\T)$ contains at least a point outside $p$. Since $p\in\tilde{E}(\T)$, $\dim \tilde{E}(\T)\ge 0$. If it has positive dimension, we are done. In case it has dimension 0, by point (\ref{bullet: count general eigenpoints}) we have $\deg\tilde{E}(\T)\ge 2^{n+1}-1>1$. Since $\tilde{E}(\T)$ is smooth at $p$, it contains at least another point.\qedhere
\end{enumerate}
\end{proof}
\end{lemma}

Given a 0-dimensional subscheme of $\p^n$ of length $2^{n+1}-1$, we can ask how to detect whether it is the eigenscheme of a cubic. Each point has to satisfy the $\binom{n+1}{2}$ equations
\begin{equation}\label{eq: minors}
\left\{ \frac{\partial f}{\partial x_i} x_j - \frac{\partial f}{\partial x_j} x_i=0 : 0 \leq i<j \leq n \right\},
\end{equation}
whose indeterminates are the $\binom{n+3}{3}$ coefficients of $f$. These conditions are linear
. In order to have a solution, the matrix associated to the system of linear equations \eqref{eq: minors} cannot have maximal rank. Hence the length $2^{n+1}-1$ subscheme is the eigenscheme of a cubic if and only if the maximal minors of the $\binom{n+3}{3} \times (2^{n+1}-1)\binom{n+1}{2}$ 
matrix vanish. Although these conditions are complicated, they provide a computational way to check if $2^{n+1}-1$ given points are the points of an eigenscheme.

\begin{dfn} Let $\T=(q_0,\dots,q_n)$ be a partially symmetric tensor. The \emph{irregular eigenscheme} of $\T$ is the subscheme $\irr(\T)\subset\p^n$ defined by the ideal $(q_0,\dots,q_n)\subset\C[x_0,\dots,x_n]$. The residue of $E(\T)$ with respect to $\irr(\T)$ is called the \emph{regular eigenscheme} and denoted by $\reg(\mathcal{T})$. As a consequence, we can compute the ideal of $\reg(\T)$ as the saturation
	\[
		I(\reg(\T))=I\left(\overline{\reg(\T)}\right)= \left(I(E(\T)):I(\irr(\T)) \right).
	\]
\end{dfn}

To clarify the terminology, the regular eigenpoints are the points $p$ such that the rational map $(q_0,\dots,q_n) \: \p^n \dashrightarrow \p^n$ is both regular at $p$ and fixes $p$.
When $\T$ is a symmetric tensor, the regular eigenpoints of the associated cubic polynomial $f$ are the fixed points of the gradient map $\nabla f:\p^n\dashrightarrow\p^n$ defined by
\[p\mapsto \left[ \frac{\partial f}{\partial x_0}(p),\dots,\frac{\partial f}{\partial x_n}(p)\right]. \]
The closed points of the irregular eigenscheme are the singular points of the hypersurface $V(f) \subset \p^n$. 
It will be useful for this paper to consider what happens to the eigenscheme under the following  group action.
	\begin{dfn} \label{def: twisted action}
		Let $U \in \operatorname{GL}_{n+1}(\C)$ and let $\mathcal{T} := (q_0(x), \ldots, q_n(x) )$ be a partially symmetric tensor. We define the \emph{twisted action} of $U$ on $\mathcal{T}$ by
			\[
				\Psi_U \mathcal{T} := ( q_0( x U ) , \ldots, q_n(x U ) ) \cdot U^{-1}.
			\]
		i.e, $U$ acts on the quadrics by change of coordinates, then $U^{-1}$ acts by taking linear combinations of the slices of the tensor.
	\end{dfn}

\begin{rmk}\label{rmk:SO invariant}Let $\rho_{\mathrm{std}}$ be the standard representation of $\GL_{n+1}(\C)$. The action described in Definition~\ref{def: twisted action} defines the representation $\rho_{\mathrm{std}}^{\otimes 2} \otimes \rho_{\mathrm{std}}^\vee$. For the subgroup $\SO_{n+1}(\C)$, the action is equivalent to acting by orthogonal change of coordinates.
By \cite[Theorem 2.20]{qi1}, the eigenscheme is $\SO_{n+1}(\C)$-invariant. The next lemma shows what happens for the action of $\GL_{n+1}(\C)$.\end{rmk}

	\begin{lem} \label{lem: extended action on eigscheme} \label{thm:SO4 invariant}
		Let $U \in \operatorname{GL}_{n+1}(\C)$ and let $\mathcal{T} := (q_0(x), \ldots, q_n(x) )$ be a partially symmetric tensor. Then 
		$
			E\left( \Psi_U \mathcal{T} \right) = U^{-1}E(\mathcal{T}).
		$
	\end{lem}
	
	\begin{proof}
The equations $\{ q_i(x) x_j - q_j(x) x_i : 0 \leq i,j \leq n \}$ vanish at $x$ if and only if the minors of
		\[
		\begin{pmatrix}
		(x_0 & x_1 & \dots & x_n) \cdot U^{-1} \\
		(q_0( x) & q_1(x) & \dots & q_n( x)) \cdot U^{-1}
		\end{pmatrix}
		\]
		also vanish. Setting $ y := (x_0, \dots, x_n) \cdot U^{-1}$, we have that the minors of
		\[
		\begin{pmatrix}
		y_0 & y_1 &\dots & y_n \\
		(q_0( y U) & q_1(y U) & \dots & q_n( y U)) \cdot U^{-1}
		\end{pmatrix}
		\]
		vanish if and only if $ y \in U^{-1}(x)$. The last system of equations defines the eigenscheme of $\Psi_UT$.
	\end{proof}

\section{Dimensions of the regular and irregular eigenschemes of cubics}\label{sec: deg and dim}
The eigenscheme of a cubic can exhibit a wide range of structure. For instance, it can be non-reduced or it can have components of different dimension.

\begin{es}\label{es:not pure dimension}Let $f=x_1(x_1x_2+x_3^2+x_0^2)$ and consider the conics $C_1=V(x_1-2x_2,x_0^2-4x_2^2+x_3^2)$ and $C_2=V(x_1,x_0^2+x_3^2)$. Explicit computations show that 
	\[
	E(f)=C_1\cup C_2 \cup \{[0,\sqrt{2},1,0],[0,-\sqrt{2},1,0]\}
	\]
	while $\irr(f)=C_2$. Hence $\reg(f)$ has components of both dimension 0 and 1.
\end{es}

\begin{es}\label{es:not reduced} Let $f=x_0(x_1^2+x_2^2+x_3^2)+x_1^3$ and consider the non-reduced curve $C$ defined by the ideal $(x_1^2,2x_0^2-x_2^2-x_3^2)$. Regular eigenpoints are dense in $C$ and $\overline{\reg(f)}\supset C$.
\end{es}

In this section, we describe some of the possibilities for the dimensions of $\reg(f)$ and $\irr(f)$.

\begin{pro} \label{pro: dimE and dimS} 
	Let $f\in\C[x_0,\dots,x_n]_3$ be a homogeneous cubic. Then
	\begin{enumerate}[(a)]
		\item $\dim \irr(f) + 1 \ge \dim \reg(f)$. In particular, $\dim\reg(f)=0$ whenever $f$ is smooth;
		\item $\dim \irr(f)=n-1$ if and only if $\irr(f)$ is a hyperplane. In this case $X$ contains a double hyperplane and $\reg(f)$ has either $0$, $1$, or $2$ closed points;
		\item if $\dim \reg(f)=n-1$, then $\dim \irr(f) = n-2$.

	\end{enumerate}
\end{pro}

\begin{proof}
\begin{enumerate}[(a)]
	\item 
		Let $H$ be the hyperplane of $\p^{n+1}=\proj\C[x_0,\dots,x_n,\lambda]$ defined by $\lambda=0$. Let $\pi:\p^{n+1}\dashrightarrow\p^n$ be the projection from the point $[0,\dots,0,1]$. Since all fibers of $\pi$ have dimension 1, we have:
		\begin{alignat*}{4}
			\dim \irr(f) &= \dim(\pi^{-1} \irr(f))-1 &&= \dim \left( \pi^{-1}E(f)\cap H \right) -1\\
			&\ge\dim(\pi^{-1}E(f))-2 &&= \dim E(f)-1\ge\dim \reg(f)-1.
		\end{alignat*}
	If $f$ is smooth, then this implies $\dim\reg(f)\le 0$. We already know that $\dim\reg(f)\neq-1$ by Lemma \ref{lem:nonempty}(\ref{bullet: at least an eigenpoint}).
	\item  Let $X = V(f)\subset\p^n$ be the projective cubic hypersurface defined by $f$.
	For every pair of points $s_1 , s_2 \in\irr(f)$, the line $\langle s_1,s_2\rangle$ intersects $X$ with multiplicity at least $4$, so it is contained in $X$ by B\'{e}zout's theorem. This means that $X$ contains the secant variety of $\irr(f)$. Assume $\dim \irr(f) = n-1$. If $\irr(f)$ was not supported on a hyperplane, then its secant variety  would be the whole $\p^n$, contradiction. The converse is clear.
	
	We now determine the number of regular eigenpoints in this case. By Remark \ref{rmk:SO invariant}, we may assume that $\irr(f) =V(x_0^r)$ for some $r \geq 2$ up to an orthogonal transformation. We can write $f=x_0^2(a_0x_0+\ldots+a_{n}x_n)$, and the eigenscheme is defined by
		\[
		\begin{rcases}
		\begin{dcases}
			2x_0(a_0x_0+\ldots+a_nx_n)+a_0 x_0^2=\lambda x_0 \\
			a_1x_0^2=\lambda x_1, \quad a_2x_0^2=\lambda x_2, \quad  \ldots, \quad a_nx_0^2=\lambda x_n
		\end{dcases}
		\end{rcases}
		. 
		\]
	If $\lambda \neq 0$, then up to scaling we may assume $\lambda=1$, so we have
		\begin{equation*}
		\begin{rcases}
		\begin{dcases}
			2x_0(a_0x_0+\ldots+a_nx_n)+a_0 x_0^2= x_0 \\
			a_1x_0^2 = x_1, \quad \ldots, \quad a_nx_0^2= x_n
		\end{dcases}
		\end{rcases}.
		\end{equation*}
	We see that ${\left((a_1^2 + \ldots + a_n^2 ) x_0^2 + 3a_0 x_0 - 1\right) x_0 = 0}$ by eliminating $x_1, \ldots, x_n$ from the first relation. The only solution when $x_0=0$ is $[0,\dots,0,1]$. The other factor has at most $2$ solutions in $x_0$, so the claim follows.
	\item 
	From (a), we see $\dim \irr(f) \geq n-2$. From part (b), $\dim \irr(f) < n-1$. \qedhere
\end{enumerate}

\end{proof}

\newcommand{\longexampleone}{ \begin{array}{rl} &\phantom{+} x_0(x_1^2 - x_2^2 - x_3^2) \\ & + (\theta x_1 + i x_2 + x_3)^3 \end{array}}

\newcommand{\longexampletwo}{ \begin{array}{rl} &\phantom{+} 3x_0(x_1^2+x_2^2) \\ &+ (x_1+ix_2)^3 \end{array}}

The bounds from Proposition \ref{pro: dimE and dimS} on the dimensions of $\irr(f)$ and $\reg(f)$ are optimal for ternary and quaternary cubics. For any $\delta, \epsilon\in \{-1, 0,1\}$ satisfying these requirements, Table~\ref{tab:dims-ternary} gives an example  of a ternary cubic $f$ such that $\dim \reg(f) = \delta$ and $\dim \irr(f) = \epsilon$. Table~\ref{tab:dims-quaternary} gives examples of quaternary cubics for any admissible $\delta, \epsilon \in \{-1, 0,1, 2 \}$.

\begin{table}[H]
	\centering
	\renewcommand*{\arraystretch}{1.4}
	\captionof{table}{Dimensions of the regular and irregular eigenschemes for plane cubics. Here, $\delta:=\dim \reg(f)$, $\epsilon:=\dim \irr(f)$, and $i$ denotes the element such that $i^2=-1$.}
	\begin{tabular}{|c|ccc|}
		\hline
		\backslashbox{$\delta$ \kern-1em}{$\!\!\epsilon$ \kern-1em}& -1&0&1 \\ \hline		
		-1 & $\emptyset$ & $3x_0(x_1^2+x_2^2) + (x_1+ix_2)^3$ & $x_0^2(x_1+ix_2)$ \\ \hline 
		0 &  $x_0^3+x_1^3+x_2^3$ & $x_0^3+x_1^3$ & $ x_0^3$ \\[1.1ex] \hline
		1 & $\emptyset$  & $x_0(x_1^2+x_2^2)$ & $\emptyset$  \\ \hline
	\end{tabular}
	 \label{tab:dims-ternary} 
\end{table}

\begin{table}[H]
	\centering
		\renewcommand*{\arraystretch}{1.4}
		\captionof{table}{Dimensions of the regular and irregular eigenschemes of cubic surfaces. Here, $\delta:=\dim \reg(f)$, $\epsilon:=\dim \irr(f)$, while $i$, respectively $\theta$ denote elements such that $i^2=-1$, respectively $\theta^6 = -8/9$.} 
		\begin{tabular}{|c|cccc|}
		\hline
		\backslashbox{$\delta$ \kern-1em}{$\!\!\epsilon$ \kern-1em}& -1&0&1& 2\\ \hline		
			-1 & $\emptyset$ & $\longexampleone$ & $\longexampletwo$ & $x_0^2(x_1+ix_2)$ \\ \hline 
			0 & $\sum_{j=0}^3 x_j^3$ & $x_0^3+x_1^3+x_2^3$ & $x_0^3+x_1^3$ & $ x_0^3$ \\[1.1ex] \hline
			1 & $\emptyset$ & $ \sum_{j=1}^3 x_0x_j^2+x_1^3$ & $x_0(x_1^2+x_2^2)$ & $\emptyset$  \\ \hline
		    2 & $\emptyset$ & $\emptyset$ & $x_0(x_1^2+x_2^2+x_3^2)$ & $\emptyset$  \\ 
		\hline
		\end{tabular}
		\label{tab:dims-quaternary} 
\end{table}

Notice that there are partially symmetric tensors whose eigenscheme has dimension $3$.
\begin{lem}\label{lem:trivial}
	Let $\mathcal{T} = (q_0, \ldots, q_n)\in\Sym^2\C^{n+1}\otimes\C^{n+1}$. Then $E(\T)=\p^n$ if and only if there is a linear form $\ell$ such that $q_i = \ell x_i$ for every $i\in\{0,\dots,n\}$.
\end{lem}

\begin{proof}Consider the matrix
	\[
			\left( \begin{matrix}
			x_0 & x_1 & \dots & x_n \\
			q_0 & q_1 & \dots & q_n
			\end{matrix}\right) \]
		as in Definition \ref{def:eigenscheme}. Assume that $q_i = \ell x_i$ for every $i\in\{0,\dots,n\}$. Then we are dealing with
	\[
			\left( \begin{matrix}
			x_0 & x_1 & \dots & x_n \\
			\ell x_0 & \ell x_1 & \dots & \ell x_n
			\end{matrix}\right) ,\]
			which has rank at most 1 for every $x\in\p^n$. Conversely, if $x_iq_j - x_jq_i$ is identically zero for every $i,j\in\{0,\dots,n\}$, then the result follows from the fact that $\C[x_0,\ldots, x_n]$ is a unique factorization domain.
\end{proof}

One can understand the regular eigenscheme of cubic cones in $\mathbb{P}^n$ by studying the eigenscheme of cubics in $\mathbb{P}^{n-1}$. In general, Lemma~\ref{lemma: reduce dimension of eigensystem} shows how examples in lower dimensions help fill in the classification of possible strata for higher dimensional tensors. If $V(f) \subset \p^n$ is a cone over a plane cubic curve, then up to a $\SO_{n+1}(\C)$ transformation we may assume that $f$ satisfies the hypothesis of Lemma~\ref{lemma: reduce dimension of eigensystem}.

\begin{lem} \label{lemma: reduce dimension of eigensystem}
	Let $f\in\C[x_0,\dots,x_n]_3$ be a homogeneous cubic such that $\frac{\partial f}{\partial x_n}=0$. Let $\phi:\p^{n-1}\rightarrow\p^n$ be the embedding of $\p^{n-1}$ as the hyperplane $x_n=0$ in $\p^n$.
Let $\tilde{f}\in\C[x_0,\dots,x_{n-1}]_3$ be the cubic defined by
	\[\tilde{f}(x_0,\dots,x_{n-1})=f(x_0\dots,x_{n-1},0).
	\]
	Then $\reg(f)=\phi(\reg(\tilde{f}))$.
	
\end{lem}

\begin{proof}
 Let $p=[p_0,\dots, p_n]\in\reg(f)$. Then
	\[
	\rk
	\begin{pmatrix}
	\frac{\partial f}{\partial x_0}(p) & \dots & \frac{\partial f}{\partial x_{n-1}}(p) & 0\\		p_0 &\dots & p_{n-1} & p_n		
	\end{pmatrix} \le 1.
	\]
	First we prove that $p_n=0$. Assume by contradiction that $p_n \neq 0$. By hypothesis all the minors vanish, so $\frac{\partial f}{\partial x_i}(p)=0$ for every $i\in\{0,\dots,n\}$ and thus $p\in\irr(f)$, contradiction. Hence $p_n=0$. By omitting the last column of the matrix above we see that the conditions defining $\reg(f)$ are the same defining the intersection of $\reg(\tilde{f})$ with the hyperplane $x_n=0$.
\end{proof}

\section{Zero-dimensional regular eigenschemes}\label{sec: zero dimensional}
Even if the general $f\in\C[x_0,\dots,x_n]_3$ has $2^{n+1}-1$ regular eigenpoints, some cubics have less. The first problem we want to tackle is whether it is possible to find a cubic with a prescribed number of regular eigenpoints. Moreover, it is interesting to check if we can realize all the regular eigenpoints on $\R$, instead of $\C$. We can answer these questions for both ternary and quaternary cubics.

\begin{thm} \label{thm: zero-dimensional-configs} Let $f\in\C[x_0,\dots,x_n]_3$. If $\dim \reg(f) \leq 0$, then $f$ has at most $2^{n+1}-1$ regular eigenpoints. Moreover
	\begin{enumerate}
	\item for every $t\in \{0,1,\dots,7\}$ there exists a ternary cubic $f$ such that $\reg(f)$ is reduced and consists of $t$ real points;
	\item for every $t\in \{0,1,\dots,15\}$ there exists a quaternary cubic $f$ such that $\reg(f)$ is reduced and consists of $t$ real points.
	\end{enumerate} 
\end{thm}

\begin{proof} From Section \ref{sec:prelim}, the eigenscheme $\tilde{E}(f)\subset\p^n$ of a general cubic consists of $2^{n+1}-1$ points. If there are more, the dimension increases by B\'ezout's theorem.
To prove the second statement, we exhibit examples in Table~\ref{tab:nbEigep}.
\end{proof}

\begin{table}[h]
	\renewcommand*{\arraystretch}{1.1}
	\captionof{table}{Examples of ternary and quaternary cubics with a prescribed number of eigenpoints. All of them are real
. The \texttt{Macaulay2} \cite{M2} and \texttt{Magma} \cite{Magma} scripts  to check the examples are available at \cite{stefana-website}.}
	\begin{minipage}[b]{0.45\linewidth}
		\begin{tabular}{|cc|}
		\hline
			$\# \reg(f)$  & $f$  \\ \hline
			$0$ & $x_0^2(x_1+ix_2)$ \\ \hline
			$1$ & $x_0^3$  \\ \hline
			$2$ & $x_1^2 x_2$  \\ \hline
			$3$ & $x_0^3+x_1^3$  \\ \hline
			$4$ & $x_0 x_1 x_2$  \\ \hline
			$5$ & $x_0^3+x_1^2 x_2$  \\ \hline
			$6$ & $x_0^2 x_1+x_0^2 x_2+x_1 x_2^2$  \\ \hline
			$7$ & $x_0^3+x_1^3+x_2^3$ \\ \hline
\end{tabular}
\end{minipage}
\begin{minipage}[b]{0.45\linewidth}
\begin{tabular}{|cc|}
\hline
$\# \reg(f)$  & $f$  \\ \hline
$8$ & $x_0^2 x_1+x_2^2 x_3$   \\ \hline
			$9$ & $x_0 x_1 x_2+x_3^3$  \\ \hline
			$10$ & $x_0 x_1 x_2+x_0 x_3^2+x_1 x_2^2$ \\ \hline
			$11$ & $x_0^3+x_1^2 x_2+3 x_3^3$  \\ \hline
			$12$ & $10x_1x_2^2-x_0^2x_1-x_0^2x_2-x_0x_3^2$  \\ \hline
			$13$ & $x_0^2 x_1+x_0^2 x_2+x_1 x_2^2+x_3^3$  \\ \hline
			$14$ & $x_0x_3^2+x_0x_1x_2+x_1^3+10x_1x_2^2+x_2^3$  \\ \hline
			$15$ & $x_0^3+x_1^3+x_2^3+x_3^3$ 	\\		 
		\hline
\end{tabular}
\end{minipage}
 \label{tab:nbEigep}
\end{table}

Many other interesting behaviours appear, such as collinear, triangular or tetrahedral configurations. 
 
\begin{es}\label{example: rank 1, rank 2 and nondegenerate rank 3}	
	The regular eigenpoints of $x_0^3+x_1^3+x_2^3$ are
	\[
	[1,0,0,0], [0,1,0,0], [0,0,1,0], [1,1,0,0], [1,0,1,0], [0,1,1,0], [1,1,1,0].
	\]
	They are coplanar, in the configuration described by Figure~\ref{picture: configuration of 7 coplanar eigenpoints}.
\end{es}

\begin{figure}[h]
	\begin{center}		
		\includegraphics*[scale=0.6]{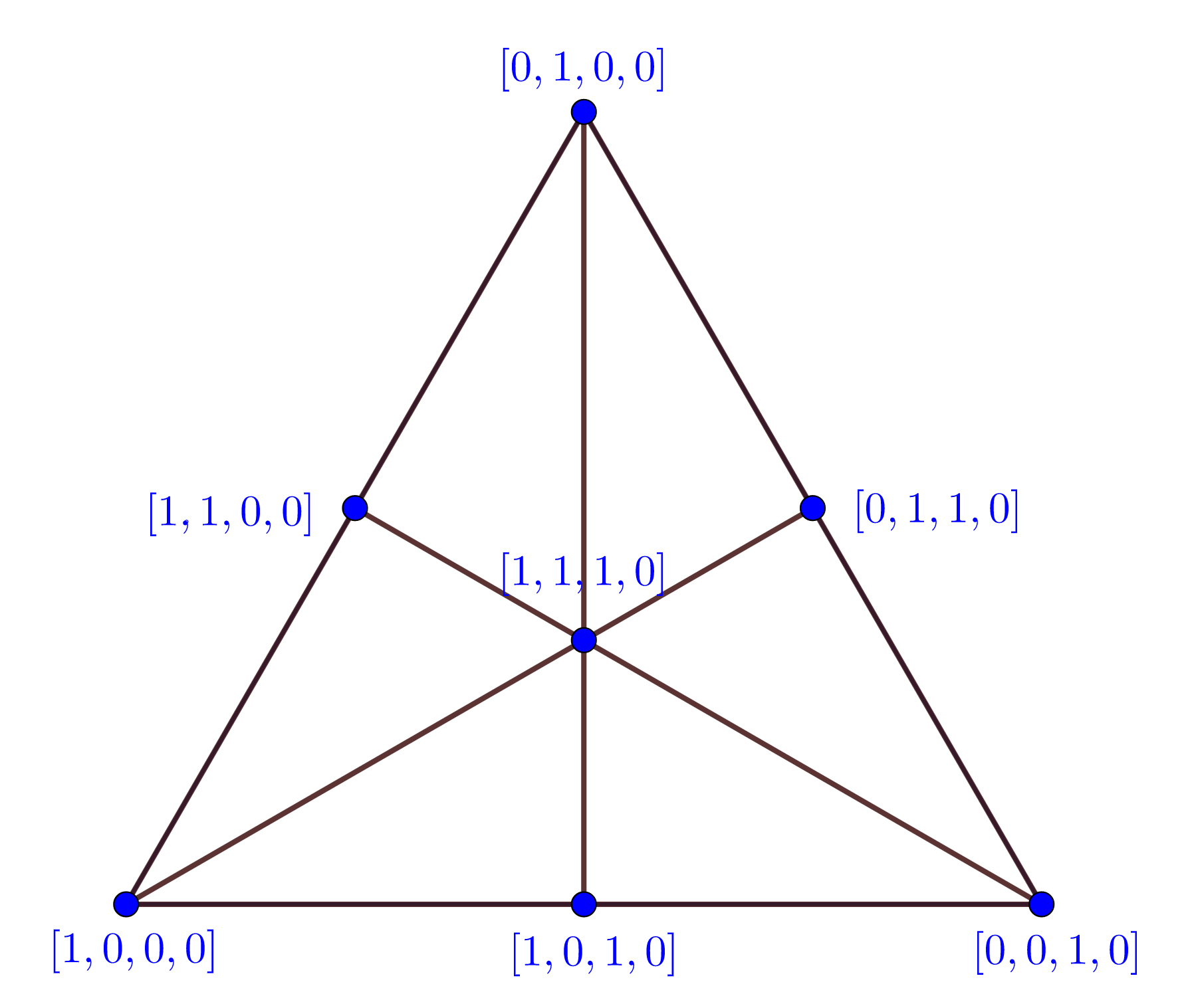}
	\end{center}
	\caption{Configuration of 7 eigenpoints.}
	\label{picture: configuration of 7 coplanar eigenpoints}
\end{figure}

\begin{es}\label{example: rank 4, smooth case}
	The regular eigenpoints of $x_0^3+x_1^3+x_2^3+x_3^3$ are
	\begin{align*}
	[1,0,0,0], [0,1,0,0], [0,0,1,0],[0,0,0,1], [1,1,0,0],\\ 
	[1,0,1,0],[0,1,1,0], [1,1,1,0], [1,0,0,1], [0,1,0,1],\\ 
	[0,0,1,1], [1,1,0,1], [1,0,1,1], [0,1,1,1], [1,1,1,1].
	\end{align*}
	They are not in general position. Each coordinate plane contains exactly 7 of these points arranged in the configuration of Figure~\ref{picture: configuration of 7 coplanar eigenpoints}.
\end{es}

\section{A Grassmannian as a parameter space of eigenschemes} \label{sec: grassmanian}
	
	We turn our attention to the problem of classifying which sets of points are the eigenpoints of a cubic homogeneous polynomial in four variables. We proceed indirectly by first studying the problem for partially symmetric tensors to highlight the underlying geometry, and recover the result for symmetric tensors by specialization. The main idea is as follows. Let $\mathcal{T} = (q_0,\ldots, q_3)$ be a partially symmetric tensor. By construction, the polynomials $q_i-\lambda x_i$ are linearly independent, so 
		\[
			H_\mathcal{T} := 
			\mathrm{Span}_{\mathbb{C}}(q_0-\lambda x_0,\ldots,q_3-\lambda x_3) \tag{$\ast$} \label{eq: 3-plane}
		\]
		is a 3-plane in $\p^{14}=\p(\C[x_0,\dots,x_3,\lambda]_2)$.	On the other hand, given a $3$-plane in $\p^{14}$, the intersection of the dual $10$-plane with the image of the Veronese $\nu_2\: \p^4 \hookrightarrow \p^{14}$ defines a subscheme of $\p^4$, which is generically $0$-dimensional and of degree $16$. We describe the conditions under which a $3$-plane in $\p^{14}$ is of the form \eqref{eq: 3-plane} in terms of the Pl\"ucker coordinates.
		
	\begin{thm} Fix the monomial basis $\{x_0^2,x_0x_1,\dots,x_3^2, x_0\lambda,\dots,x_3\lambda,\lambda^2\}$ of\\
	$\C[x_0,\dots,x_3,\lambda]_2$. The morphism $\Sym^2 \C^{4} \otimes \C^{4}\longrightarrow \Gr(3,\p^{14})$ given by $\mathcal{T} \mapsto [H_\mathcal{T}]$ is an isomorphism onto its image. In the Pl\"ucker coordinates with respect to this basis, this image is the subscheme of the Grassmannian defined by the following conditions:
		\begin{enumerate}
			\item all the entries of the column of $H_\mathcal{T}$ corresponding to $\lambda^2$ are zero, and
			\item
				the Pl\"ucker coordinate corresponding to the columns of $H_\mathcal{T}$ labelled by $\{ x_0\lambda, x_1 \lambda, x_2 \lambda, x_3\lambda\}$ is non-zero.
		\end{enumerate}
		If $\nu_2\: \p^4 \hookrightarrow \p^{14}$ is the Veronese embedding, then $\nu_2(\widetilde E(\mathcal{T})) = H_\mathcal{T}^\vee \cap \nu_2(\p^4)$.
	\end{thm}
		
	\begin{proof}
		If $\mathcal{T}=(q_0,q_1,q_2,q_3)$ is a partially symmetric tensor, then the $3$-plane $H_\mathcal{T} := \mathrm{Span}_{\mathbb{C}}(q_0-\lambda x_0,\ldots,q_3-\lambda x_3)$ satisfies conditions (1) and (2). We show that the morphism, when restricted to the image, has an inverse. A $3$-plane in $\mathrm{Gr}(3,\p^{14})$ satisfying conditions (1) and (2) is a subspace represented by a $4\times 15$ matrix
		\[
		\kbordermatrix{
			& x_{0}^2 & x_{0}x_{1} & \ldots & x_{3}^2 & x_{0}\lambda& x_{1}\lambda & x_{2}\lambda & x_{3}\lambda& \lambda^2\\
			& m_{1,1} & m_{1,2} & \ldots & m_{1,10} & m_{1,11}&m_{1,12}&m_{1,13}&m_{1,14}&m_{1,15} \\
			& m_{2,1} & m_{2,2} & \ldots & m_{2,10} & m_{2,11}&m_{2,12}&m_{2,13}&m_{2,14}&m_{2,15} \\
			& m_{3,1} & m_{3,2} & \ldots & m_{3,10} & m_{3,11}&m_{3,12}&m_{3,13}&m_{3,14}&m_{3,15}  \\
			& m_{4,1} & m_{4,2} & \ldots & m_{4,10} & m_{4,11}&m_{4,12}&m_{4,13}&m_{4,14}&m_{4,15} 
		}.
		\]	
		By the hypothesis the entries in the column labelled by $\lambda^2$  are zero and the $4\times 4$~block
		\[
		\kbordermatrix{
			&  x_{0}\lambda& x_{1}\lambda & x_{2}\lambda & x_{3}\lambda\\
			&  m_{1,11}&m_{1,12}&m_{1,13}&m_{1,14} \\
			& m_{2,11}&m_{2,12}&m_{2,13}&m_{2,14} \\
			& m_{3,11}&m_{3,12}&m_{3,13}&m_{3,14} \\
			&  m_{4,11}&m_{4,12}&m_{4,13}&m_{4,14} 
		}
		\]
		is invertible
		, we can apply the reduced row echelon form to get
		\[
		\kbordermatrix{
			& x_{0}^2 & x_{0}x_{1} & \ldots & x_{3}^2 & x_{0}\lambda& x_{1}\lambda & x_{2}\lambda & x_{3}\lambda& \lambda^2\\
			& \tilde{m}_{1,1} & \tilde{m}_{1,2} & \ldots & \tilde{m}_{1,10} & -1&0&0&0&0 \\
			& \tilde{m}_{2,1} & \tilde{m}_{2,2} & \ldots & \tilde{m}_{2,10} & 0&-1&0&0&0 \\
			& \tilde{m}_{3,1} & \tilde{m}_{3,2} & \ldots & \tilde{m}_{3,10} & 0&0&-1&0&0  \\
			& \tilde{m}_{4,1} & \tilde{m}_{4,2} & \ldots & \tilde{m}_{4,10} & 0&0&0&-1&0
		}.
		\]This is a $3$-plane given by a tensor. The statement about $\widetilde{E}(\mathcal{T})$ is clear.
	\end{proof}


	\begin{cor} \label{cor: symmetric eigenschemes as 3-planes}
		The parameter space of eigenschemes of symmetric tensors is an open subvariety of a linear subspace of $\Gr(3,\p^{14})$. 
	\end{cor}

	\begin{proof}
		
		Let 
		\[
		\kbordermatrix{
			& x_{0}^2 & x_{0}x_{1} & \ldots & x_{3}^2 & x_{0}\lambda& x_{1}\lambda & x_{2}\lambda & x_{3}\lambda& \lambda^2\\
			& m_{1,1} & m_{1,2} & \ldots & m_{1,10} & -1&0&0&0&0 \\
			& m_{2,1} & m_{2,2} & \ldots & m_{2,10} & 0&-1&0&0&0 \\
			& m_{3,1} & m_{3,2} & \ldots & m_{3,10} & 0&0&-1&0&0  \\
			& m_{4,1} & m_{4,2} & \ldots & m_{4,10} & 0&0&0&-1&0
		}
		\]
		be the affine coordinates for a $3$-plane coming from a symmetric tensor. The $4$ quadrics corresponding to the rows are of the form $\frac{\partial f}{\partial x_0} - \lambda x_0, \ldots, \frac{\partial f}{\partial x_3} - \lambda x_3$ if and only if the $m_{i,j}$, interpreted as coefficients of the quadrics $q_0,\ldots, q_3$, satisfy the linear conditions $\{ \frac{\partial q_i}{\partial x_j} - \frac{\partial q_j}{\partial x_i} = 0 : i < j \}$. These give linear relations of the Pl\"ucker coordinates by \cite[Proposition~3.1.2]{GKZ2008discriminants}.
	\end{proof}
	
	As shown in \cite[Section 4]{AboSeigalSturmfels}, the eigendiscriminant is a homogeneous polynomial of degree $96$ in the entries of a tensor, that vanishes whenever the eigenscheme has a point of multiplicity greater than or equal to $2$ or the eigenscheme is positive dimensional. It is interesting to compare the eigendiscriminant to the Hurwitz form of the image of $\nu_2 \: \p^4 \hookrightarrow \p^{14}$, which is a polynomial in the Pl\"ucker cooordinates for $\Gr(3,\p^{14})$ that vanishes on the $3$-dimensional planes that intersect the Veronese tangentially~\cite[Theorem 1.1]{Sturmfels2016hurwitz}. As the definitions of these two polynomials are closely related, we make the following conjecture.
	
	\begin{conj}\label{conjHurwitz}
		Restricted to the $3$-planes coming from eigenschemes, the eigendiscriminant divides the Hurwitz form. 
	\end{conj}
	
	We can verify Conjecture~\ref{conjHurwitz} for the eigenschemes of binary cubic forms. In this case, the eigenscheme defines a line in $\p^5=\p(\C[x_0,x_1,\lambda]_2)$. We consider the Hurwitz form of the image of $\nu_2 \: \p^2 \hookrightarrow \p^{5}$ in the coordinate ring $\Gr(1,\p^5)$. This is a polynomial of degree 6 in Pl\"ucker coordinates. If we express an element in $\Gr(1,\p^5)$ as the image of the transpose of a $2\times 6$ matrix, the primal Pl\"ucker coordinates are the maximal minors of the $2\times 6$ matrix.  We compute the Hurwitz form in the primal Pl\"ucker coordinates in the \texttt{Macaulay2} package "Resultants". Now, we restrict it to the eigenscheme as follows. Suppose we are given a binary cubic form $f=a_0x_0^3+a_1x_0^2x_1+a_2x_0x_1^2+a_3x_1^3$. The line associated to $f$ is given by the image of the transpose of the following $2\times 6$ matrix

\[
		\begin{pmatrix}
		3a_0 & 2a_1 &-1 & a_2 & 0 & 0 \\
		a_1 & 2a_2 & 0 & 3a_3 & -1 & 0
		\end{pmatrix}.
		\]

The primal Stiefel coordinates are the entries, and the primal Pl\"ucker coordinates are the $2 \times 2 $ minors of this matrix. Substituting the Pl\"ucker coordinates into the Hurwitz form we have 
\begin{align*}
&36a_0^2a_1^2+32a_1^4-108a_0^3a_2-156a_0a_1^2a_2+216a_0^2a_2^2+61a_1^2a_2^2-144a_0a_2^3+32a_2^4 \\ &-108a_0^2a_1a_3-144a_1^3a_3+306a_0a_1a_2a_3-156a_1a_3^2a_3+81a_0^2a_3^2+216a_1^2a_3^2\\ & \qquad \qquad \qquad \qquad-108a_1a_2a_3^2+36a_2^2a_3^2-108a_1a_3^3.
\end{align*}

On the other hand, we compute the eigendiscriminant in terms of the coordinates in $\p^5$ \cite[Example 4.4]{AboSeigalSturmfels}. In this case, the eigendiscriminant and the Hurwitz form are equal up to a factor of $(-1)$. The accompanying \texttt{Macaulay2} \cite{M2} script is available at the link~\cite{stefana-website}.

\section*{Acknowledgments}

The authors thank Anna Seigal and Simon Telen for helpful discussions. We would like to thank Marta Panizzut, Kristian Ranestad, and Emre Sert\"oz for their organization of the cubic surfaces day at the University of Oslo, as well as for organizing the intermediate meetings. Finally, we extend our gratitude to Bernd Sturmfels for inspiring the project. This article is based on Question 16 of the website \url{http://cubics.wikidot.com}.

\nocite{*}
\bibhere


\providecommand{\bysame}{\leavevmode\hbox to3em{\hrulefill}\thinspace}
\providecommand{\MR}{\relax\ifhmode\unskip\space\fi MR }
\providecommand{\MRhref}[2]{%
  \href{http://www.ams.org/mathscinet-getitem?mr=#1}{#2}
}
\providecommand{\href}[2]{#2}
\begin{thebibliography}{10}

\bibitem{AboSeigalSturmfels}
Hirotachi Abo - Anna Seigal - Bernd Sturmfels, \emph{Eigenconfigurations of
  Tensors}, Algebraic and Geometric Methods in Discrete Mathematics,
  Contemporary Mathematics, vol. 685, American Mathematical Society,
  Providence, RI, 2017, pp.~1--25.

\bibitem{Magma}
Wieb Bosma - John Cannon - Catherine Playoust, \emph{The {M}agma algebra
  system. {I}. {T}he user language}, Journal of Symbolic Computation {24}
  no.~3-4 (1997), 235--265, \url{http://dx.doi.org/10.1006/jsco.1996.0125}.

\bibitem{stefana-website}
Turku~O. Celik - Francesco Galuppi - Avinash Kulkarni - Miruna-Stefana Sorea,
  \emph{Macaulay2 and Magma scripts accompanying this article},
  \url{https://sites.google.com/view/mirunastefanasorea/code}.

\bibitem{book3264}
David Eisenbud - Joe Harris, \emph{3264 and all that---a second course in
  algebraic geometry}, Cambridge University Press, Cambridge, 2016.

\bibitem{GKZ2008discriminants}
Israel~M. Gelfand - Michail~M. Kapranov - Andrei~V. Zelevinsky,
  \emph{Discriminants, resultants and multidimensional determinants}, Modern
  Birkh\"{a}user Classics, Birkh\"{a}user Boston, Inc., Boston, MA, 2008.

\bibitem{M2}
Daniel~R. Grayson - Michael~E. Stillman., \emph{Macaulay2, a software system
  for research in algebraic geometry},
  \url{http://www.math.uiuc.edu/Macaulay2/}.

\bibitem{landsberg}
Joseph~M. Landsberg, \emph{Tensors: geometry and applications}, Graduate
  Studies in Mathematics, vol. 128, American Mathematical Society, Providence,
  RI, 2012.

\bibitem{lim}
Lek-Heng Lim, \emph{Singular values and eigenvalues of tensors: a variational
  approach}, Proceedings of the IEEE International Workshop on Computational
  Advances in Multi-Sensor Adaptive Processing (CAMSAP '05) (2005), 129--132.

\bibitem{qi05}
Liqun Qi, \emph{Eigenvalues of a real supersymmetric tensor}, Journal of
  Symbolic Computation {40} no.~6 (2005), 1302--1324.

\bibitem{qi07}
Liqun Qi, \emph{Eigenvalues and invariants of tensors}, Journal of Mathematical
  Analysis and Applications {325} no.~2 (2007), 1363--1377,
  \url{https://doi.org/10.1016/j.jmaa.2006.02.071}.

\bibitem{qi1}
Liqun Qi - Ziyan Luo, \emph{Tensor analysis: Spectral theory and special
  tensors}, Philadelphia, PA: Society for Industrial and Applied Mathematics,
  2017.

\bibitem{Segre1942nonsingular}
Beniamino Segre, \emph{The Non-singular Cubic Surfaces}, Oxford University
  Press, Oxford, 1942.

\bibitem{Seigal2018ranks}
Anna Seigal, \emph{Ranks and Symmetric Ranks of Cubic Surfaces},
  arXiv:1801.05377 (2018).

\bibitem{Sturmfels2016hurwitz}
Bernd Sturmfels, \emph{{The Hurwitz form of a projective variety}}, Journal of
  Symbolic Computation {79} no.~1 (2017), 186--196,
  \url{http://dx.doi.org/10.1016/j.jsc.2016.08.012}.

\end{thebibliography}
\end{document}